\documentclass[12pt,a4paper]{article}
\usepackage[OT1]{fontenc}
\usepackage[english]{babel}
\usepackage{amsfonts}
\usepackage{amsthm}

\usepackage{graphicx}

\def\ddd{ {\bf \Delta}}
\def\h{ {\cal H} }
\def\a{ {\cal A} }

\def\l{ {\cal L} }
\def\n{ {\cal N} }
\def\v{ {\cal V} }

\def\b{ {\cal B} }

\def\m{ {\cal M} }

\def\e{ {\cal E} }
\def\o{ {\cal O} }

\def\c{ {\cal C} }

\def\j{ {\cal J} }

\def\DD{\mathbb{D}}

\newtheorem{teo}{Theorem}[section]
\newtheorem{prop}[teo]{Proposition}
\newtheorem{lem}[teo]{Lemma}
\newtheorem{coro}[teo]{Corollary}

\theoremstyle{definition}
\newtheorem{rem}[teo]{Remark}

\title{The C$^*$-algebra of a composition reflection.}
\author{Esteban Anduchow}

\begin{document}

\maketitle 

\begin{abstract}
We study the  C$^*$ algebra generated by the composition operator $C_a$ acting on the Hardy space $H^2$ of the unit disk, given by $C_af=f\circ\varphi_a$, where 
$$
\varphi_a(z)=\frac{a-z}{1-\bar{a}z},
$$ 
for $|a|<1$. Also several operators related to $C_a$ are examined.
 \end{abstract}

\bigskip



\section{Introduction}

Let $\mathbb{D}=\{z\in\mathbb{C}:|z|\le 1\}$ be the unit disk and $\mathbb{T}=\{z\in\mathbb{C}: |z|=1\}$ the unit circle. For $a\in\mathbb{D}$, consider the analytic automorphism which maps $\mathbb{D}$ onto $\mathbb{D}$  
$$
\varphi_{a}(z)=\frac{a-z}{1-\bar{a}z},
$$
The fact that $\varphi_a(\varphi_a(z))=z$ implies  that the composition operator 
\begin{equation}\label{Ca}
C_{\varphi_a} f=f\circ\varphi_a.
\end{equation}
induced by this automorphisms is a  {\it reflection} (i.e.,  satisfies  that $C_a^2=I$) in  $H^2=H^2(\DD)$,  the Hardy space of the disk (see \cite{cowenmccluer}).
In this note we characterize the C$^*$-algebra $\c^*(C_a)$ generated by $C_a$, which is also the C$^*$-algebra generated by two projections, the ortogonal projections onto the two eigenspaces of $C_a$: $N(C_a-I)$ and $N(C_a+I)$. We profit from the vast bibliography on this subject, specially  the results by G.K Pedersen \cite{pedersen}, and the excellent survey \cite{spitkovskyetal}.

We also consider several operators related to $C_a$. Among these, the symmetry (=selfadjoint reflections) $\rho_a$, obtained from the polar decomposition of $C_a$, and $W_a=M_{\psi_a}C_a$, where $\psi_a=\frac{(1-|a|^2)^{1/2}}{1-\bar{a}z}$ is the normalized Szego kernel.

The characterization of $\c^*(C_a)$ requires the study of the spectrum of the product of the projections onto $N(C_a\pm I)$.

The contents of the paper are the following. In Section 2 we introduce preliminay facts and notations. In Section 3 we study elementary relations between the spectra of of $PQ$ and $P\pm Q$, for $P,Q$ orthogonal projections. In Section 4 we apply these result to the case of $P=P_{N(C_a-I)}$ and $Q=P_{N(C_a+I)}$. In the short Section 5 we use these data to characterize $\c^*(C_a)$, using the results of G.K. Pedersen \cite{pedersen} (via the exposition done in \cite{spitkovskyetal}). The automorphism $\varphi_a$ induces also a reflection in $L^2(\mathbb{T})$, which we call $\Gamma_a$; in Section 6 we study the relation between $C_a$ and $\Gamma_a$. In Section 7 we study the mentioned symmetries $\rho_a$ and $W_a$. In Section 8 we consider the relation of $C_a$ with the Toeplitz isometry $T_{\varphi_{\omega_a}}$, where $\omega_a$ is the unique fixed point of $\varphi_a$ inside $\mathbb{D}$.

\section{Preliminaries and notations}
Note that $C_0f(z)=f(-z)$. Denote by $\e$ the subspace of {\it even } functions in $H^2$:
$$\e=\{f\in H^2: f(z)=f(-z)\}=N(C_0-I).
$$
Its  orthogonal complement is the space $\o$ of {\it odd } functions, $\o=N(C_0+I)$.
Denote by  $\omega_a$  the unique fixed point of $\varphi_a$ inside the disk:
$$
\omega_a:=\frac{1}{\bar{a}}\{1-\sqrt{1-|a|^2}\} \ \hbox{ if } a\ne 0, \ \hbox{ and } \omega_0=0.
$$
The fixed point $\omega_a$ is useful in describing the two eigenspaces of $C_a$ for $a\ne 0$. In \cite{disco y composicion} it was shown that
$$
N(C_a-I)=C_{\omega_a}(\e) \ \hbox{ and } \ N(C_a+I)=C_{\omega_a}(\o).
$$
These assertions follow in a direct manner from the elementary identity
\begin{equation}\label{phis}
\varphi_{\omega_a}\circ\varphi_a=-\varphi_{\omega_a}.
\end{equation}
In general ($a\ne 0$), $C_a$ is non-selfadjoint.
It was shown by Cowen \cite{cowen} (see also \cite{cowenmccluer}) that
$$
C_a^*=(C_{\varphi_a})^*=M_{\frac{1}{1-\bar{a}z}} C_a (M_{1-\bar{a}z})^*=M_{\frac{1}{1-\bar{a}z}} C_a T_{1-a\bar{z}},
$$
where, for $g\in L^\infty(\mathbb{T})$,  $M_g$ and $T_g$ denote, respectively, the multiplication and Toeplitz operators with symbol $g$. Equivalently,
\begin{equation}\label{cea*}
C_a^*=M_{\frac{1}{1-\bar{a}z}} C_a-a M_{\frac{1}{1-\bar{a}z}} C_a S^*,
\end{equation}
where $S^*=(M_z)^*$  (or co-shift) is the adjoint of the shift operator $S=M_z$. 
Then
\begin{equation}
C_aC_a^*=\frac{1}{1-|a|^2}(I-\bar{a}S)(I-aS^*)=\frac{1}{1-|a|^2}T_{1-\bar{a}z}T_{1-a\bar{z}}.
\end{equation}
On the other hand, it can be shown that
\begin{equation}\label{toeplitz}
C_a^*C_a=(1-|a|^2)T_{\frac{1}{|1-\bar{a}z|^2}},
\end{equation}
which also equals (see Proposition 7.5 in \cite{douglas})
\begin{equation}
C_a^*C_a=(1-|a|^2)T_{\frac{1}{1-a\bar{z}}}T_{\frac{1}{1-\bar{a}z}}.
\end{equation}
\begin{rem}\label{normas}
Using (\ref{toeplitz}), one easily obtains that the spectrum of $C_a^*C_a$ is 
$$
\sigma(C_a^*C_a)=\sigma(T_\frac{1-|a|^2}{|1-\bar{a}z|})=[\frac{1-|a|}{1+|a|},\frac{1+|a|}{1-|a|}],
$$
the image of $\mathbb{T}$ under the function $\frac{1-|a|^2}{|1-\bar{a}z|}$ (see \cite{douglas}). Since $C_a$ is invertible, we also have $\sigma(C_aC_a^*)$ equals this interval.  In particular, 
$$
\|C_a\|=\|C_a^*\|=\|C_a^*C_a\|^{1/2}=\frac{1+|a|}{1-|a|},
$$
which is well know (see \cite{cowenmccluer}).
Similarly, we obtain the spectrum and the norm of the commutator $[C_a^*,C_a]$ (note that $C_aC_a^*=(C_a^*C_a)^{-1}$):
$$
[C_a^*,C_a]=C_a^*C_a-(C_a^*C_a)^{-1}=f(C_a^*C_a),
$$
where $f(t)=t-\frac{1}{t}$ (considered in $(0,+\infty)$). Then
$$
\sigma([C_a^*,C_a])=f\left([\frac{1-|a|}{1+|a|},\frac{1+|a|}{1-|a|}]\right)=[\frac{-4|a|}{1-|a|^2}, \frac{4|a|}{1-|a|^2}],
$$
and also $\|[C_a^*,C_a]\|=\frac{4|a|}{1-|a|^2}$.
\end{rem}
\section{Spectral relations between $PQP$ and $P\pm Q$}
In this section we collect several elementary (certainly well known) results concerning the spectra of $PQP$ and $P\pm Q$ for pairs of orthogonal projections $P,Q$.
First we state properties concerning eigenvalues.  Note that $PQP$ is a positive contraction, and that $P-Q$ is a selfadjoint contraction. Chandler Davis \cite{davis} observed that the spectrum of $P-Q$ is symmetric with respect to the oirigin, in the following sense. Denote  $A=P-Q$ and put $\h'=(N(A-I)\oplus N(A+I))^\perp$. Then it is elementary that $\h'=(R(P)\cap N(Q) \oplus N(P)\cap R(Q))^\perp$ reduces both $P$ and $Q$. Denote by $P'$ and $Q'$ (and $A'=P'-Q'$)  the corresponding reductions. Then there exists a symmetry $V$ of $\h´$ such that $ VP'V=Q'$ (and therefore also $VQ'V=P'$). Thus $VA'V=-A'$, and in particular the spectrum of $A'$ is symmetric: $\lambda\in\sigma(A')$ iff $-\lambda\in\sigma(A')$, and the multiplicity function is symmetric. It follows that the spectrum of $A$ has the same property, save for the eventual eigenvalues $\pm 1$, where this symmetry could break.

 Let us state the following basic  properties. 

\begin{lem}\label{lema 54}
Suppose that $P$ and $Q$ are orthogonal projections acting in $\h$.
\begin{enumerate}
\item
 If $PQP$ has a an eigenvalue $\lambda\ne 0, 1$, then $\pm(1-\lambda^2)^{1/2}$ are eigenvalues of $P-Q$.
\item
Conversely, if $\mu\ne0,1,-1$ is an eigenvalue of $P-Q$, then $1-\mu^2$ is an eigenvalue for $PQP$
\end{enumerate}
\end{lem}
\begin{proof}
Suppose $PQP f=\lambda f$ for $\lambda\ne 0,1$ and $\|f\|=1$. Then $f\in R(P)$ and therefore the subspace $\v$ generated  by $f$ and $Q f$  is invariant  for $P$ and $Q$:
$$
P f=f \ ; \ PQ f=PQP f=\lambda f \ \ \hbox{ and } \ \ Q f
$$ 
belong to $\v$. Then $P-Q$ is a selfadjoint operator acting in $\v$, which is two dimensional (if $Q f$ where a multiple of $f$, then either $Q f=0$ and then $P QP f=PQ f=0$; or $Q f =\alpha f$ and thus $f\in R(P)\cap R(Q)$ and therefore  $PQ P f=f$). 
Then $\{f, Q f\}$ is a basis for $\v$ and the matrices of $P$, $Q$ and $P-Q$ as operators in $\v$ for this basis are, respectively:
$$
\left( \begin{array}{cc} 1 & \lambda \\ 0 & 0 \end{array} \right) , \ \left( \begin{array}{cc} 0 & 0 \\ 1 & 1 \end{array} \right) \ \hbox{ and } \left( \begin{array}{cc} 1 & \lambda \\ -1 & -1 \end{array} \right).
$$
Note that $\lambda\in(0,1)$: $PQ P\ge 0$. The eigenvalues of the third matrix are $-(1-\lambda^2)^{1/2}$ and $(1-\lambda^2)^{1/2}$.

For the second statement, suppose that $Pg-Qg=\mu g$, for $\mu\ne 0, \pm 1$, i.e., \begin{equation}\label{uno}
Qg=Pg-\mu g.
\end{equation}
Then, applying $Q$ one has  $Qg=QPg-\mu Qg$, i.e. $Qg=\frac{1}{\mu+1}g$. Then, substituting this identity in (\ref{uno}), we get
$$
\frac{1}{\mu+1}QPg=Pg-\mu g, 
$$
and applying $P$: $\frac{1}{\mu+1}PQPg=(1-\mu)Pg$, i.e.
$$
PQP g=(1-\mu^2)Pg.
$$
Note that $Pg\ne 0$: $Pg=0$ would imply $Qg=-\mu g$, i.e. $\mu=0$ or $\mu=-1$.
 \end{proof}
\begin{rem}\label{bobo}
Note that if  $f\ne 0$,   $PQPf=f$  if and only if $f\in R(P)\cap R(Q)$. Sufficiency is trivial. Necessity: pick $\|f\|=1$; $PQPf=f$ implies $f\in R(P)$, and thus $PQf=f$. Then 
$$
1=\langle PQf,f\rangle=\langle Qf,Pf\rangle=\langle Qf,f\rangle,
$$
which clearly implies $f\in R(Q)$.
\end{rem}   
The following result can be verified along the same lines as the above lemma. We include the elementary proof.
\begin{lem}\label{lema 55}
Let $P, Q$ be orthogonal projections. Then $\lambda$ is an eigenvalue of $P-Q$ with $|\lambda|<1$ if and only if $1\pm (1-\lambda^2)^{1/2}$ are eigenvalues of $P+Q$.
\end{lem}
\begin{proof}
Let $\lambda$ be an eigevalue of $P-Q$.
First suppose that $\lambda=0$, and $(P-Q)f=0$ for $f\ne 0$. Then  $f\in R(P)\cap R(Q)\oplus N(P)\cap N(Q)$. Thus $f=f_1+f_0$ with $Pf_1=f_1=Qf_1$ and $Pf_0=0=Qf_0$. Then $(P+Q)f_1=2f_1$ and $(P+Q)f_0=0$, i.e. $1\pm (1-0)^{1/2}$ are eigenvalues of $P+Q$. 

Suppose now that $(P-Q)f=\lambda f$ with $\lambda\ne -1, 0, 1$. Consider, as in the proof above,  the subspace $\v$ generated by $f$ and $Qf$.  Note that $\{f, Qf\}$ are linearly independent: if $Qf=\alpha f$, then either $\alpha=0$ or $\alpha=1$. If $\alpha=0$, then $\lambda f=(P-Q)f=Pf$ (and $\lambda\ne 0$) imply $\lambda=1$ (a contradiction); if $\alpha=1$, then $\lambda f=(P-Q)f=Pf-f$ implies $Pf=(1+\lambda)f$, i.e. $\lambda=-1$ (again a contradiction). Thus $\{f,Qf\}$ is a basis for $\v$. 
Note
that $(P+Q)f=(P-Q)f+2Qf=\lambda f+2 Qf$. On the other hand, $(P-Q)f=\lambda f$ implies that $Pf=\lambda f+Qf$, and thus $Pf=P(\lambda f+Qf)=\lambda Pf+PQf$ and thus $PQf=(1-\lambda)Pf=(1-\lambda)(\lambda f+Qf)=(\lambda-\lambda^2)f+(1-\lambda)Qf$.
Then $(P+Q)Qf=PQf+Qf=(\lambda-\lambda^2)f+(2-\lambda)Qf$.
Therefore $(P+Q)\v\subset\v$ and its matrix in the basis $\{f,Qf\}$ is 
$$
\left(\begin{array}{cc} \lambda & \lambda-\lambda^2 \\ 2 & 2-\lambda \end{array}\right),
$$
 whose eigenvalues are $1\pm (1-\lambda^2)^{1/2}$. 

The converse is similar.
\end{proof}
Now we focus on arbitrary spectral values, not necessarily eigenvalues. To this effect, we shall need P. Halmos \cite{halmos} results on pairs of subspaces / projections. Fix $P,Q$ orthogonal projections. consider the folowing natural orthogonal decomposition of $\h$:
$$
R(P)\cap R(Q)\  \oplus \ N(P)\cap N(Q)\ \oplus \ R(P)\cap N(Q) \ \oplus \ N(P)\cap R(Q) \ \oplus \ \h_0.
$$
The space $\h_0$ is usually called the generic part of $P$ and $Q$. Clearly this decomposition reduces both $P$ and $Q$. Note that in this decomposition,
$$
P=I \ \oplus \ 0 \ \oplus \  I \ \oplus \ 0 \ \oplus \ P_0 \ \  \hbox{ and } \ \ 
Q= I \ \oplus \ 0 \ \oplus \ 0 \ \oplus \ I \ \oplus \ Q_0.
$$
So that 
$$
PQP= I\ \oplus \ 0 \ \oplus \ 0\ \oplus \ 0 \ \oplus \ P_0Q_0P_0,
$$
$$
P-Q=0\ \oplus \ 0\ \oplus \ I\ \oplus \ -I\ \oplus \ P_0-Q_0.
$$
and 
$$
P+Q=2\ \oplus \ 0\ \oplus \ I\ \oplus \ I\ \oplus \ P_0+Q_0.
$$
Therefore, in order to analyze the spectra of $PQP$, $P-Q$ and $P+
Q$ we need to focus on the operators acting in the generic part $\h_0$.

Continuing with Halmos' theory, he proved that there exists a unitary isomorphim between $\h_0$ and a product Hilbert space $\l\times \l$, and a positive operator $X$ acting in $\l$, with $\|X\|\le\pi/2$ and $N(X)=\{0\}$, such that the projections $P_0$ and $Q_0$ are carried to
$$
P_0\simeq\left(\begin{array}{cc} I & 0 \\ 0 & 0 \end{array}\right) \ \hbox{ and } \ Q_0\simeq \left(\begin{array}{cc} C^2 & CS \\ CS & S^2 \end{array}\right),
$$
where $C=\cos(X)$ and $S=\sin(X)$.

Note in particular that $\sigma(X)\subset[0,\pi/2]$.

\begin{lem}\label{espectros}
Let $P_0$ and $Q_0$ as above and let $\lambda\ne 0$. Then 
\begin{enumerate}
\item
$$
\lambda\in\sigma(P_0Q_0P_0) \iff \pm\sqrt{1-\lambda^2}\in\sigma(P_0-Q_0).
$$
\item
$$
\lambda\in\sigma(P_0Q_0P_0) \iff 1\pm \lambda^{1/2}\in\sigma(P_0+Q_0).
$$
\end{enumerate}
\end{lem}
\begin{proof}
1. We can reason with the operators in $\l\times\l$. Then  
$$
P_0Q_0P_0\simeq\left(\begin{array}{cc} C^2 & 0 \\ 0 & 0 \end{array}\right),
$$
and thus $\sigma(P_0Q_0P_0)=\{\cos^2(t): t\in\sigma(X)\}\cup\{0\}$.
Also
$$
P_0-Q_0\simeq\left(\begin{array}{cc} S^2 & -CS \\ -CS & -S^2 \end{array}\right)=\left(\begin{array}{cc} S & 0 \\ 0 & S \end{array}\right)\left(\begin{array}{cc} S & -C \\ -C & -S \end{array}\right),
$$
where the last two matrices commute (we have used that $C$ and $S$ commute). Note that the second matrix is a symmetry:
$$
\left(\begin{array}{cc} S & -C \\ -C & -S \end{array}\right)^*=\left(\begin{array}{cc} S & -C \\ -C & -S \end{array}\right) \ \hbox{ and }\ \left(\begin{array}{cc} S & -C \\ -C & -S \end{array}\right)^2=\left(\begin{array}{cc} I & 0 \\ 0 & I \end{array}\right).
$$
It is well known that if $a,b$ are elements of a C$^*$-algebra such that  $ab=ba$, then $\sigma(ab)\subset \{\lambda\mu: \lambda\in\sigma(a) , \mu\in\sigma(b)\}$. Therefore, since $\sigma\left(\left(\begin{array}{cc} S & -C \\ -C & -S \end{array}\right)\right)=\{-1; 1\}$, we have that 
$$
\sigma(P_0-Q_0)\subset \{\pm \lambda: \lambda\in\sigma\left(\begin{array}{cc} S & 0  \\ 0 & S \end{array}\right)\}=\{\pm \sin(t): t\in\sigma(X)\}.
$$
Conversely, suppose that $P_0-Q_0-\lambda 1$ is invertible. Since
$$
P_0-Q_0-\lambda 1\simeq\left(\begin{array}{cc} S & -C \\ -C & -S \end{array}\right)\{\left(\begin{array}{cc} S & 0 \\ 0 & S \end{array}\right) -\lambda\left(\begin{array}{cc} S & -C \\ -C & -S \end{array}\right)\},
$$
this means that $\left(\begin{array}{cc} S & 0 \\ 0 & S \end{array}\right) -\lambda\left(\begin{array}{cc} S & -C \\ -C & -S \end{array}\right)$ is invertible. Thus the square of this operator
$$
\left(\begin{array}{cc} S^2 & 0 \\ 0 & S^2 \end{array}\right) -2\lambda\left(\begin{array}{cc} S & -C \\ -C & -S \end{array}\right) +\lambda^2 \left(\begin{array}{cc} S^2 & 0 \\ 0 & S^2 \end{array}\right)
$$
is positive and invertible. Therefore the diagonal entries are positive and invertible, i.e.,
$S^2\pm2\lambda S+\lambda^2I=(S\pm\lambda I)^2$ is invertible. That is, $\lambda\ne\pm \sin(t)$ for $t\in\sigma(X)$.

2. The proof is similar. Note that 
$$
P_0+Q_0\simeq\left(\begin{array}{cc} I+C^2 & CS \\ CS & S^2\end{array} \right)=\left(\begin{array}{cc} I & 0 \\ 0 & I\end{array} \right) + \left(\begin{array}{cc} C^2 & CS \\ CS & -C^2\end{array} \right)
$$
and that 
$$
\left(\begin{array}{cc} C^2 & CS \\ CS & -C^2\end{array} \right)=\left(\begin{array}{cc} C & 0 \\ 0 & C\end{array} \right)\left(\begin{array}{cc} C & S \\ S & -C\end{array} \right).
$$
The right hand matrices commute, and the matrix $\left(\begin{array}{cc} C & S \\ S & -C\end{array} \right)$ is also a symmetry. Thus, with the same argument as above, we have that 
$$
\sigma\left(\left(\begin{array}{cc} C^2 & CS \\ CS & -C^2\end{array} \right)\right)\subset \{\pm \cos(t): t\in\sigma(X)\}.
$$
Also note that 
$$
\left(\begin{array}{cc} C^2 & CS \\ CS & -C^2\end{array} \right)-\lambda\left(\begin{array}{cc} I & 0 \\ 0 & I\end{array} \right)=
\left(\begin{array}{cc} C & S \\ S & -C\end{array} \right)\{ \left(\begin{array}{cc} C & 0 \\ 0 & C\end{array} \right)-\lambda \left(\begin{array}{cc} C & S \\ S & -C\end{array} \right)\}.
$$
This product is invertible if and only if the right hand factor is invertible, which implies that its square is positive and invertible, and therefore the diagonal entries of this square are invertible, i.e. $(C\pm\lambda I)^2$ are invertible. 
Therefore
$$
\sigma(P_0+Q_0)=\sigma\left( \left(\begin{array}{cc} I & 0 \\ 0 & I\end{array} \right) + \left(\begin{array}{cc} C^2 & CS \\ CS & -C^2\end{array} \right)\right)=\{1\pm \cos(t): t\in\sigma(X)\}.
$$ 
\end{proof}

 \section{Projections onto the eigenspaces of $C_a$}
Recall from \cite{ando} the formulas for the orthogonal projections onto the range and the null space of an oblique projection $Q$:
$$
P_{R(Q)}=Q(Q+Q^*-I)^{-1} \ \hbox{ and } \ P_{N(Q)}=(I-Q)(I-Q-Q^*)^{-1}.
$$
For the oblique projection $\frac12(I-C_a)$, whose range and nullspace are, respectively,  the (non orthogonal) eigenspaces $N(C_a-I)$ and $N(C_a+I)$, we have:
\begin{equation}\label{formulas proyectores}
P_{N(C_a-I)}=(I+C_a)(C_a+C_a^*)^{-1} \ \hbox{ and } \ P_{N(C_a+I)}=(C_a-I)(C_a+C_a^*)^{-1}.
\end{equation}

Denote by 
$$
\ddd_a:=P_{N(C_a-I)}P_{N(C_a+I)}P_{N(C_a-I)}.
$$
We shall compute  the norm $\|P_{N(C_a-I)}P_{N(C_a+I)}\|=\|\ddd_a\|^{1/2}$ below, and further study the spectrum of $\ddd_a$.

The following result will be useful:
\begin{teo}\label{cuenta P+Q}
Let $a\in\mathbb{D}$. Then 
$$
\sigma(P_{N(C_a-I)}+P_{N(C_a+I)})=[1-|a|, 1+|a|].
$$
Moreover, none of these spectral values are eigenvalues.
\end{teo}
\begin{proof}
It follows from (\ref{formulas proyectores}) that 
$$
P_{N(C_a-I)}+P_{N(C_a+I)}=2C_a(C_a+C_a^*)^{-1}.
$$
Note that the inverse of the  element $C_a(C_a+C_a^*)^{-1}$ is (using (\ref{toeplitz}))
$$
(C_a+C_a^*)C_a=I+C_a^*C_a=I+(1-|a|^2)T_{\frac{1}{|1-\bar{a}z|^2}}.
$$
The spectrum of $T_{\frac{1}{|1-\bar{a}z|^2}}$ is the image of its symbol \cite{douglas}:
$$
\sigma(T_{\frac{1}{|1-\bar{a}z|^2}})=[\frac{1}{(1+|a|)^2},\frac{1}{(1-|a|)^2}].
$$
Then 
$$
\sigma((I+C_a^*C_a)^{-1})=[\frac{1-|a|}{2},\frac{1+|a|}{2}]
$$
and therefore
$$
\sigma(P_{N(C_a-I)}+P_{N(C_a+I)})=[1-|a|, 1+|a|].
$$
An eigenvalue of $P_{N(C_a-I)}+P_{N(C_a+I)}$ would yield an eigenvalue of $T_{\frac{1}{|1-\bar{a}z|^2}}$.
\end{proof}
In particular, we have that $\|P_{N(C_a-I)}+P_{N(C_a+I)}\|=1+|a|$.
We can compute the norm of $\ddd_a$:
\begin{prop}
The operator $\ddd_a$ has no (non nil) eigenvalues in $H^2$. Moreover.
$$
\|\ddd_a\|=|a|^2.
$$
\end{prop} 
\begin{proof}
Let $\lambda\ne 0$ be an eigenvalue of $\ddd_a$:
$\ddd_a f=\lambda f$, for $\|f\|=1$. Since clearly $\|\ddd_a\|\le 1$, it must be $|\lambda|\le 1$. Note that $\ddd_af=f$ would imply (see Remark \ref{bobo}) that $f\in N(C_a-I)\cap N(C_a+I)=\{0\}$. Thus $\lambda<1$. 

Since $f\in R(P_{N(C_a-I)})$,
$$
P_{N(C_a-I)}P_{N(C_a+I)}^\perp P_{N(C_a-I)}f=(P_{N(C_a-I)}-P_{N(C_a-I)}P_{N(C_a+I)}P_{N(C_a-I)})f
$$
$$
=(I-P_{N(C_a-I)}P_{N(C_a+I)}P_{N(C_a+I)})f=(1-\lambda)f.
$$
Then using Lemma \ref{lema 54} with $P=P_{N(C_a-I)}$ and $Q=P_{N(C_a+I)}^\perp$, we get that  $\pm\left(1-(1-\lambda)^2\right)^{1/2}$ are eigenvalues of 
$$ 
P-Q^\perp=P_{N(C_a-I)}-P_{N(C_a+I)}^\perp=P_{N(C_a-I)}+P_{N(C_a+I)}-I.
$$
But it was shown in Theorem \ref{cuenta P+Q} that $P_{N(C_a-I)}+P_{N(C_a+I)}$ has no eigenvalues.

From Theorem \ref{cuenta P+Q} we know also that $\sigma(P_{N(C_a-I)}+P_{N(C_a+I)})=[1-|a|, 1+|a|]$. Therefore, $\|P_{N(C_a-I)}+P_{N(C_a+I)}\|=1+|a|$. In \cite{duncan taylor} J. Duncan and P.J. Taylor proved that if $P,Q$ are non nil projections, then 
$$
\|P+Q\|=1+\|PQ\|.
$$
Then $\|P_{N(C_a-I)}P_{N(C_a+I)}\|=|a|$.
 \end{proof}

We can determine the spectrum of $\ddd_a$ (in particular, obtain another proof of  $\|\ddd_a\|=|a|^2$). We shall use Theorem \ref{espectros}. Therefore, it will be useful to compute the position of $N(C_a-I)$ and $N(C_a+I)$. First note that since these subspaces are complementary,
$$
N(C_a-I)\cap N(C_a+I)=\{0\}
$$
and 
$$
N(C_a-I)^\perp\cap N(C_a+I)^\perp=\langle N(C_a-I)+N(C_a+I)\rangle^\perp=\{0\}.
$$
In \cite{disco y composicion} it was shown (Prop. 6.1) that
\begin{equation}\label{61}
\dim N(C_a-I)\cap N(C_a+I)^\perp =1 \ \hbox{ but } \ N(C_a-I)^\perp\cap N(C_a+I)=\{0\}.
\end{equation}
Therefore
 in Halmos decomposition of $H^2$ in terms of $N(C_a-I)$ and $N(C_a+I)$ we have only two non trivial subspaces:
$$
N(C_a-I)\cap N(C_a+I)^\perp\oplus \h_0
$$
Denote $\ddd_a=0\oplus \ddd_a^0$, where $\ddd_a^0$ is the reduction of $\ddd_a$ to $\h_0$. Clearly $\sigma(\ddd_a)=\sigma(\ddd_a^0)\cup\{0\}$.

\begin{prop}
$\sigma(\ddd_a)=[0,|a|^2]$.
\end{prop}
\begin{proof}
Due to the observation above, we have to compute the spectrum of the reduction $\ddd_a^0$. Denote by $P_0, Q_0$ the reductions of $P_{N(C_a-I)}$ and $P_{N(C_a+I)}$ to $\h_0$. The spectrum of $P_0+Q_0$ is obtained from
$$
\sigma(P_{N(C_a-I)}+P_{N(C_a+I)})=\{1\}\cup \sigma(P_0+Q_0).
$$
Since $\sigma(P_{N(C_a-I)}+P_{N(C_a+I)})=[1-|a|, 1+|a|]$, and $1-|a|<1$, we have that also
$$
\sigma(P_0+Q_0)=[1-|a|, 1+|a|].
$$
Recall from Theorem \ref{espectros} that $\mu\in\sigma(P_0+Q_0)$ if and only if $\lambda=(\mu-1)^2\in\sigma(P_0Q_0P_0)$. Then 
the spectrum of $\ddd_a^0$ is the image of the function $f(\mu)=(\mu-1)^2$ in the interval $[1-|a|, 1+|a|]$, i.e., $[0,|a|^2]$.
\end{proof}

It is known (and elementary to verify) that
$$
N(C_a-I)\cap N(C_a+I)^\perp=N(P_{N(C_a-I)}-P_{N(C_a+I)}-I)
$$
and
$$
N(C_a-I)^\perp\cap N(C_a+I)=N(P_{N(C_a-I)}-P_{N(C_a+I)}+I).
$$
Then using again the intersections computed in (\ref{61}), 
we have that 
\begin{prop}
$\|P_{N(C_a-I)}-P_{N(C_a+I)}\|=1$.
\end{prop}
\begin{proof}
If $f\in N(C_a-I)\cap N(C_a+I)^\perp$ with $\|f\|=1$, then $(P_{N(C_a-I)}-P_{N(C_a+I)})f=f$.
\end{proof}
\section{The C$^*$-algebra generated by $C_a$.}
The $C^*$-algebra $\c^*(C_a)$ generated by $C_a$ coincides with the C$^*$-algebra generated by the projections $P_{N(C_a-I)}$ and $P_{N(C_a+I)}$. Therefore, by the results of G.K. Pedesersen \cite{pedersen} it is completely characterized by the spectrum of $\ddd_a$ (see for instance the excellent survey \cite{spitkovskyetal} on which we will base our exposition). 
Following the notation in \cite{spitkovskyetal}, given the subspaces $\l=N(C_a-I)$ and $\n=N(C_a+I)$, in order to characterize $\c_a^*$ we need the subspaces
$$
\m_{00}=\l\cap\n, \ \m_{01}=\l\cap\n^\perp, \ \m_{01}=\l^\perp\cap\n, \ \m_{11}=\l^\perp\cap\n^\perp,
$$
and $\h_0$ the orthogonal complement of the sum of the former four.
We already noticed that  $\m_{00}=\m_{11}=\m_{10}=\{0\}$ and $\dim \m_{01}=1$.
Again using  Halmos' theory \cite{halmos}, we have that $\h_0\simeq\j\times\j$ and there exists a positive operator $X\in\b(\j)$, $X\le \pi/2$, $N(X)=\{0\}$, such that the isomorphism that carries $\h_0$ onto $\j\times\j$ maps the projections $P_\l=P_{N(C_a-I)}$ and $P_\n=P_{N(C_a+I)}$ onto 
$$
P_{N(C_a-I)}\simeq\left(\begin{array}{cc} I & 0 \\ 0 & 0 \end{array}\right) \ \hbox{ and } \ P_{N(C_a+I)}\simeq\left(\begin{array}{cc} C^2 & CS \\ CS & S^2 \end{array}\right),
$$
where $C=\cos(X)$ and $S=\sin(X)$. Then $\c^*(C_a)$ can be described in terms of $H=\sin(X)^2$, or more precisely, in terms of the spectrum of $H$.
Since 
 $\sigma(\ddd_a)=[0,|a|^2]$, i.e., $\sigma(X)=[\arccos(|a|), \pi/2]$. It follows that $\sigma(H)=[1-|a|^2,1]$. Then according to Theorem 4.1 in \cite{spitkovskyetal}, we have that
\begin{teo}
Let $a\in\mathbb{D}$, $a\ne 0$. Then $\c_a^*$ is $*$-isomorphic to
$$
\{(\alpha, \left(\begin{array}{cc} f_{00}(H) & f_{01}(H)  \\ f_{10}(H) & f_{11}(H) \end{array}\right)): \alpha\in\mathbb{C}, f_{ij}\in C(1-|a|^2,1), f_{00}(1)=\alpha, f_{01}(1)=f_{10}(1)=0\}.
$$
\end{teo}
With this description, it is clear that
\begin{coro}
Let $a,b\in\mathbb{D}\setminus\{0\}$. Then $\c^*(C_a)\simeq\c^*(C_b)$.
\end{coro}
\section{Relationship between $C_a$ and $\Gamma_a$}

Since $\varphi_a$ is also a homeomorphism in $\mathbb{T}$, it induces a  composition operator $\Gamma_a$  in $L^2(\mathbb{T})$. This operator is also  reflection, and is easier to handle. For instance, its adjoint is easier to compute.

Clearly $\Gamma_a\big|_{H^2}=C_a$. Moreover, $P_{N(\Gamma_a-I)}$ leaves $H^2$ invariant. Indeed, if $f\in H^2$ and $f=f_++f_-$ with $f_+\in N(\Gamma_a-I)$ and $f_-\in N(\Gamma_a+I)$, then $\Gamma_af=f_+-f_-\in H^2$. Then $f+\Gamma_af=2f_+\in H^2$, i.e., $f_+,f_-\in H^2$. Then 
\begin{equation}\label{restricciones}
P_{N(\Gamma_a-I)}\big|_{H^2}=P_{N(C_a-I)} \ \hbox{ and } P_{N(\Gamma_a+I)}\big|_{H^2}=P_{N(C_a+I)}.
\end{equation}

Denote by $P_+$ the orthogonal projection of $L^2(\mathbb{T})$ onto $H^2$, and by $H_-:=L^2(\mathbb{T})\ominus H^2$. The fact that $\Gamma_a\big|_{H^2}=C_a$, means $P_+\Gamma_a P_+=\Gamma_a P_+=C_a$. On the other hand, by a change of variables argument it is easy to see that
$$
\Gamma_a^*=(1-|a|^2)M_{\frac{1}{|1-\bar{a}z|^2}}\Gamma_a=(1-|a|^2)M_{\frac{1}{1-\bar{a}z}\frac{z}{z-a}}\Gamma_a.
$$
If $f,g\in H^2$,
$$
\langle \Gamma_a^*f,g\rangle=\langle f,\Gamma_ag\rangle=\langle f, C_ag\rangle=\langle C_a^*f,g\rangle,
$$
i.e., $P_+\Gamma_a^*P_+=C^*_a$. Let us see how $\Gamma_a^*$ acts on $H^2$ (i.e., let us compute $P_+^\perp\Gamma_a^*P_+$).
\begin{lem}
Let $h\in H^2$, $a\in\mathbb{D}$, $a\ne 0$, and denote by $h_0=h-h(0)$. Then $\frac{h_0(\varphi_a(z))}{z-a}\in H^2$ and
$$
\Gamma_a^*h=(1-|a|^2)\frac{z}{1-\bar{a}z} \frac{h_0(\varphi_a(z))}{z-a}+h(0)\Gamma_a^*(1).
$$
Moreover, 
$$
\Gamma_a^*(1)=\frac{1}{1-\bar{a}z}+\frac{a}{z-a},
$$
where the first summand lies in $H^2$ and the second in $H_-$.
\end{lem}
\begin{proof}
If $h=h_0+h(0)$, then 
$$
\Gamma_a^*h=\Gamma_a^* h_0+h(0)\Gamma_a^*(1)=(1-|a|^2)M_{\frac{1}{1-\bar{a}z}\frac{z}{z-a}}h_0(\varphi_a(z))+h(0)\Gamma_a^*(1).
$$
Note that $h_0(\varphi_a(a))=h_0(0)=0$, and thus $\frac{h_0(\varphi_a(z))}{z-a}\in H^2$.
On the other hand, if $n\ge 0$,
$$
\langle z^n,\Gamma_a^*(1)\rangle=\langle \Gamma_a(z^n),1\rangle=\langle  \left(\frac{a-z}{1-\bar{a}z}\right)^n,1\rangle=\left(\frac{a-z}{1-\bar{a}z}\right)^n\big|_{z=0}=a^n,
$$
i.e. $\langle\Gamma_a^*(1),z^n\rangle =\bar{a}^n$,  for $n\ge 0$, and thus (for $z\in\mathbb{T}$)
$$
P_+(\Gamma_a^*(1))=\sum_{n=0}^\infty \bar{a}^nz^n=\frac{1}{1-\bar{a}z}.
$$
For $m<0$, 
$$
\langle \Gamma_a^*(1). z^m\rangle=\langle 1,\left(\frac{a-z}{1-\bar{a}z}\right)^m\rangle=\frac{1}{2\pi}\int_{\mathbb{T}} \left(\frac{\bar{a}-\bar{z}}{1-a\bar{z}}\right)^m dz=\frac{1}{2\pi}\int_{\mathbb{T}}\left(\frac{\bar{a}-1/z}{1-a/z}\right)^m dz
$$
$$
=\frac{1}{2\pi}\int_{\mathbb{T}}\left(\frac{\bar{a}-z}{1-\bar{a}z}\right)^{-m} dz=\langle \varphi_a(z)^{-m},1\rangle=a^{-m}.
$$
Then 
$$
P_+^\perp(\Gamma_a^*(1))=\sum_{m<0} a^{-m}z^{m}=\sum_{m<0}\left(\frac{a}{z}\right)^{-m}=\frac{a/z}{a-a/z}=\frac{a}{z-a}.
$$
\end{proof}
In other words, $P_+^\perp \Gamma_a^*P_+$ is the rank one operator
\begin{equation}
P_+^\perp \Gamma_a^*P_+ f=\frac{a}{z-a} \langle f  , 1\rangle. 
\end{equation}
Therefore, to complete the $2\times 2$ matrix of $\Gamma_a$ in terms of the decomposition $L^2(\mathbb{T})=H^2\oplus H_-$ it remains to compute $P_+^\perp\Gamma_a P_+^\perp$. Denote by $V$ the symmetry of $L^2(\mathbb{T})$ given by
$$
Vf(z)=f(\bar{z}).
$$
Clearly $V^*=V^{-1}=V$, and $V$ maps $H_-$ onto $H^2\ominus\langle 1\rangle$. Also it is clear that if $g\in H_-$, then 
$$
\Gamma_a(g)=VC_{\bar{a}}Vg.
$$
Indeed, if $g=\sum_{k=1}^\infty a_k z^{-k}$, then 
$$
VC_{\bar{a}}Vg=VC_{\bar{a}}(\sum_{k=1}^\infty a_kz^k)=V(\sum_{k=
1}^\infty a_k\left(\frac{\bar{a}-z}{1-az}\right)^k)=\sum_{k=1}^\infty a_k\left(\frac{\bar{a}-1/z}{1-a/z}\right)^k
$$
$$
=\sum_{k=1}^\infty a_k\left(\frac{a-z}{1-\bar{a}z}\right)^{-k}=\Gamma_ag.
$$
In particular, note that $\Gamma_ah_-\subset H_-\oplus \langle 1\rangle$.
For $g\in H_-$, the component of $\Gamma_ah$ in $\langle 1 \rangle$ is (accordingly)
$$
\langle \Gamma_ag,1\rangle=\langle VC_{\bar{a}}Vh,1\rangle=\langle C_{\bar{a}}Vg. V1\rangle=\langle C_{\bar{a}}Vg, 1\rangle,
$$
recall that for $f\in H^2$, $\langle C_bf,1\rangle=f(b)$, the quantity above equals
$$
Vg(\bar{a})=\langle Vg, k_{\bar{a}}\rangle=\langle g, Vk_{\bar{a}}\rangle= \langle g, \frac{z}{z-a}\rangle.
$$
Then, we have
\begin{teo}
The $2\times 2$ matrix of $\Gamma_a$ in terms of the decomposition $L^2(\mathbb{T})=H^2\oplus H_-$ is 
$$
\left( \begin{array}{ccc} C_a &  &\langle \ \ , \frac{a}{z-a} \rangle 1 \\ & & \\ 0 &  & VC_{\bar{a}}V-\langle \ \ , \frac{a}{z-a} \rangle 1 \end{array}\right).
$$
\end{teo}
\begin{proof}
The adjoint of $\langle \ \ , 1\rangle \frac{a}{a-z}$ is $\langle  \ \ , \frac{a}{a-z}\rangle 1$.
\end{proof}

\section{Two symmetries}

We shall consider two symmetries (i.e., selfadjoint unitaries) which are closely related to $C_a$. The first one comes from the polar decomposition of $C_a$:
$$
C_a=\rho_a|C_a|=\rho_a(C_a^*C_a)^{1/2}, \ \hbox{ i.e., } \ \rho_a=C_a(C_a^*C_a)^{-1/2}.
$$
In \cite{cpr} it was shown that the unitary part in the polar decomposition of a reflection  is in fact a symmetry. Thus,  $\rho_a^*=\rho_a^{-1}$. Note then that $C_a^*=|C_a|\rho_a$ and thus
\begin{equation}
\rho_a C_a \rho_a=C_a^*.
\end{equation}
The second symmetry can be found  in the book \cite{cowenmccluer}  of C. Cowen and B. McCluer (Exercise 2.1.9): $W_a:H^2\to H^2$, 
\begin{equation}\label{Wa}
W_af=\sqrt{1-|a|^2} M_{k_a} C_af, 
\end{equation}
where $k_a(z)=\frac{1}{1-\bar{a}z}$ is the Szego kernel, i.e.,
$$
W_af (z)=\frac{\sqrt{1-|a|^2}}{1-\bar{a}z} f(\varphi_a(z)).
$$
Cowen and McCluer mention that $W_a$ is isometric for all $p$ norms, it is easy to see that  $W_a^2=I$.
Denote by $f_a$ the $H^\infty$ function
\begin{equation}\label{fa}
f_a:=\frac{1+\bar{\omega}_az}{1-\bar{\omega}_az}.
\end{equation}
Note that $\bar{f}_a=\frac{z+\omega_a}{z-\omega_a}$ 
Then
\begin{prop}
Let $a\in\mathbb{D}$, and $\omega_a$ the fixed point of $\varphi_a$ inside $\mathbb{D}$. 
Then
$$
W_{\omega_a}C_aW_{\omega_a}=T_{f_a}C_0=C_0T_{1/f_a}.
$$
\end{prop}
\begin{proof}
By direct computation:
$$
W_{\omega_a}C_aW_{\omega_a}f(z)=\sqrt{1-|\omega_a|^2}W_{\omega_a}C_a\left(f(\varphi_{\omega_a}(z)) k_{\omega_a}(z)\right)
$$
$$
=\sqrt{1-|\omega_a|^2}W_{\omega_a}\left(f(\varphi_{\omega_a}(\varphi_a(z))) k_{\omega_a}(\varphi_a(z))\right).
$$
Recall from (\ref{phis}) that $\varphi_{\omega_a}\circ\varphi_a=-\varphi_{\omega_a}$. Thus, the above expression equals
$$
(1-|\omega_a|^2)f(-\varphi_{\omega_a}(\varphi_{\omega_a}(z))) k_{\omega_a}(z) k_{\omega_a}(\varphi_a(\varphi_{\omega_a}(z)))=(1-|\omega_a|^2)f(-z) k_{\omega_a}(z) k_{\omega_a}(\varphi_a(\varphi_{\omega_a}(z)))
$$
Note that $\varphi_{\omega_a}\circ\varphi_a=-\varphi_{\omega_a}$ and $\varphi_{\omega_a}\circ\varphi_{\omega_a}(z)=z$ imply that $\varphi_{\omega_a}\circ\varphi_a\circ\varphi_{\omega_a}(z)=-z$, and thus
$$
\varphi_a\circ\varphi_{\omega_a}(z)=\varphi_{\omega_a}(-z).
$$
Then
$$
k_{\omega_a}(\varphi_a(\varphi_{\omega_a}(z)))=\frac{1}{1-\bar{\omega}_a\left(\frac{\omega_a+z}{1+\bar{\omega}_az}\right)}=\frac{1+\bar{\omega}_az}{1-|\omega_a|^2}.
$$
Therefore
$$
W_{\omega_a}C_aW_{\omega_a}f(z)=f(-z)k_{\omega_a}(z) (1+\bar{\omega}_az)=f_a(z)f(-z)=T_{f_a}C_0f(z).
$$
The facts that $(W_{\omega_a}C_aW_{\omega_a})^2=I$, and that $f_a$ is invertible in $H^\infty$ imply that 
$$
T_{f_a}C_0=(T_{f_a}C_0)^{-1}=C_0T_{f_a}^{-1}=C_0T_{1/f_a}.
$$
\end{proof}

\begin{rem}
Since $f_a, 1/f_a\in H^\infty$, we could write $M_{f_a}, M_{1/f_a}$ instead of $T_{f_a}, T_{1/f_a}$. Also note that 
$$
W_{\omega_a}C^*_aW_{\omega_a}=\left(W_{\omega_a}C_aW_{\omega_a}\right)^*=\left(T_{f_a}C_0\right)^*=C_0T_{\bar{f}_a}.
$$
And similarly, 
$$
W_{\omega_a}C^*_aW_{\omega_a}=(C_0T_{1/f_a})^*=T_{1/\bar{f}_a}C_0.
$$
\end{rem}
In Lemma 4.1 of \cite{disco y composicion} it was shown that the (unitary) product $W_a\rho_a$ commutes with $C_a^*C_a$ (and thus also with its inverse $C_aC_a^*$). As a consequence we obtain the following:
\begin{prop}
With the current notations we have that
$$
W_a\rho_a C_a=C_a(W_a\rho_a)^*, \ W_a\rho_a C_a^*=C_a^*(W_a\rho_a)^*,
$$
and 
$$
C_a^*C_aW_a =W_a(C_a^*C_a)^{-1}.
$$
\end{prop} 
\begin{proof}
Since $W_a\rho_a$ commutes with $C_a^*C_a$, it also commutes with its square root $|C_a|$ (and therefore also $\rho_aW_a=(W_a\rho_a)^*$ commutes with $|C_a|$). Then
$$
W_a\rho_aC_a=W_a\rho_a\rho_a|C_a|=W_a|C_a|=\rho_a(\rho_a W_a)|C_a|=\rho_a|C_a|\rho_aW_a=C_a\rho_aW_a.
$$
Thus, aking adjoints,
$$
C_a^*\rho_aW_a=(W_a\rho_aC_a)^*=(C_a\rho_aW_a)^*=W_a\rho_aC_a^*.
$$
Finally, using that $\rho_a$ inertwines $C_a$ with $C_a^*$,
$$
C_a^*C_aW_a\rho_a=W_a\rho_a C_a^*C_a=W_aC_a\rho_aC_a=W_aC_aC_a^*\rho_a,
$$
and thus $C_a^*C_aW_a=W_aC_aC_a^*=W_a(C_a^*C_a)^{-1}$.
\end{proof}
\section{The relationship with $T_{\varphi_{\omega_a}}$}

Since $\varphi_{\omega_a}$ is an inner function, the Toeplitz operator $T_{\varphi_{\omega_a}}$ is an isometry. It is easy to see that it has co-rank $1$. Below we compute the orthogonal projection $T_{\varphi_{\omega_a}}T^*_{\varphi_{\omega_a}}=T_{\varphi_{\omega_a}}T_{\overline{\varphi}_{\omega_a}}$.
We shall use  the following computation:
\begin{lem}\label{cuentita}
Let $f\in H^2$ and $b\in\mathbb{D}$. 
Then 
$$
P_+(f/\varphi_b)=\frac{f-f(b)}{\varphi_b} +f(b)\bar{b}.
$$
\end{lem}
\begin{proof}
Note that 
$f/\varphi_b=\frac{f-f(b)}{\varphi_b} +\frac{f(b)}{\varphi_b}$, where the first summand belongs to $H^2$. Also note that $\frac{1}{b-z}\in H_-$, and thus 
$$
P_+(f(b)/\varphi_b)=f(b) P_+(\frac{1-\bar{b}z}{b-z})=f(b) P_+(\frac{1-|b|^2}{b-z}+\bar{b})=f(b) \bar{b}.
$$
\end{proof}
Therefore, $I-T_{\varphi_{\omega_a}}T_{\overline{\varphi}_{\omega_a}}$ 
 is the orthogonal projection onto the line generated by $k_{\omega_a}$:
using Lemma \ref{cuentita}
$$
T_{\varphi_{\omega_a}}T_{\overline{\varphi}_{\omega_a}}f=T_{\varphi_{\omega_a}}P_+(f/\varphi_{\omega_a})=\varphi_{\omega_a}\left(\frac{f-f(\omega_a)}{\varphi_{\omega_a}}+f(\omega_a)\bar{\omega}_a\right)=f-f(\omega_a)+\bar{\omega}_af(\omega_a)\varphi_{\omega_a}
$$
$$
=f+\langle f, k_{\omega_a}\rangle(1-\bar{\omega}_a\varphi_{\omega_a}).
$$
Note that $1-\omega_a\varphi_{\omega_a}=\frac{1-|\omega_a|^2}{1-\bar{\omega}_az}$, and thus the computation above equals
$$
f - \langle f, k_{\omega_a}\rangle  (1-|\omega_a|^2)k_{\omega_a}=f-\langle f, (1-|\omega_a|^2)^{1/2}k_{\omega_a}\rangle (1-|\omega_a|^2)^{1/2}k_{\omega_a}=f-\langle f, \psi_a\rangle \psi_a,
$$
where $\psi_a=(1-|\omega_a|^2)^{1/2}k_{\omega_a}$ is the normalization of $k_{\omega_a}$.

Note the folllowing facts:
\begin{prop}\label{anticonmutacion}
Let $a\in\mathbb{D}$. The Toeplitz operator $T_{\varphi_{\omega_a}}$ satisfies that
$$
T_{\varphi_{\omega_a}}C_a+C_aT_{\varphi_{\omega_a}}=0
$$
and 
$$
T_{\varphi_{\omega_a}}C_a^*+C_a^*T_{\varphi_{\omega_a}}=2 \omega_a\langle \ \ , 1\rangle k_{\omega_a}.
$$
\end{prop}
\begin{proof}
The first assertion is a direct computation:
$$
T_{\varphi_{\omega_a}}C_af=P_+\left(\varphi_{\omega_a}f(\varphi_a)\right)=\varphi_{\omega_a}f(\varphi_a),
$$
whereas
$$
C_a T_{\varphi_{\omega_a}}f=C_a (\varphi_{\omega_a}f)=\varphi_{\omega_a}(\varphi_a) f(\varphi_a),
$$
and the assertion follows recalling from (\ref{phis}) that $\varphi_{\omega_a}\circ\varphi_a=-\varphi_{\omega_a}$. 

With respect to the second assertion, using Lemma \ref{cuentita}
$$
T_{1/\varphi_{\omega_a}}C_af=P_+\left(1/\varphi_{\omega_a}f(\varphi_a)\right)=\frac{f(\varphi_a)-f(\omega_a)}{\varphi_{\omega_a}}+ \bar{\omega}_a f(\omega_a).
$$
On the other hand, similarly as above
$$
C_a T^*_{\varphi_{\omega_a}}f=C_aP_+\left(1/\varphi_{\omega_a} f\right)=C_a\left(\frac{f(z)-f(\omega_a)}{\varphi_{\omega_a}}+\bar{\omega}_af(\omega_a)\right).
$$
Since $C_a(1)=1$ and again using (\ref{phis}), we get
$$
T_{1/\varphi_{\omega_a}}C_af=-\frac{f(\varphi_a)-f(\omega_a)}{\varphi_{\omega_a}}+ \bar{\omega}_a f(\omega_a).
$$
Therefore
$$
(T_{\varphi_{\omega_a}}^*C_a+C_aT_{\varphi_{\omega_a}}^*)f=2\bar{\omega}_af(\omega_a)=2\bar{\omega}_a\langle f , k_{\omega_a}\rangle 1,
$$
i.e.,
$$
T_{\varphi_{\omega_a}}^*C_a+C_aT_{\varphi_{\omega_a}}^*=2\bar{\omega}_a\langle \ \ , k_{\omega_a}\rangle 1,
$$
and thus
$$
T_{\varphi_{\omega_a}}C_a^*+C_a^*T_{\varphi_{\omega_a}}=2\omega_a \langle \ \ , 1\rangle k_{\omega_a}.
$$
\end{proof}
\begin{rem}
It follows form the previous lemma that $T_{\varphi_{\omega_a}}^2=T_{\varphi^2_{\omega_a}}$ commutes with $C_a$. Whereas
$$
T^2_{\varphi_{\omega_a}}C_a^*=T_{\varphi_{\omega_a}}\left(-C_a^*T_{\varphi_{\omega_a}}+2\omega_a\langle \ \ , 1\rangle k_{\omega_a}\right)=C_a^*T_{\varphi_{\omega_a}}^2+3\omega_a\{\langle\ \ , 1\rangle T_{\varphi_{\omega_a}}k_{\omega_a}-\langle \ \ , T_{\varphi_{\omega_a}}^*1\rangle k_{\omega_a}\}.
$$
Note that $T_{\varphi_{\omega_a}}k_{\omega_a}=\frac{\omega_a-z}{(1-\bar{\omega}_az)^2}$ and $T_{\varphi_{\omega_a}}^*1=\bar{\omega}_a$ (constant function).
In particular $T_{\varphi_{\omega_a}}^2$ commutes with $C_a^*$ modulo a rank two operator.
\end{rem}
\begin{rem}
Another consequence of the above lemma is that the isomery $T_{\varphi_{\omega_a}}$ intertwines $C_a$ with $-C_a$:
$$
T_{\varphi_{\omega_a}}^*C_aT_{\varphi_{\omega_a}}f=T_{\varphi_{\omega_a}}^*\left(-\varphi_{\omega_a} f(\varphi_a)\right)=-P_+\left((1/\varphi_{\omega_a}) \varphi_{\omega_a} f(\varphi_a)\right)=-f(\varphi_a),
$$
i.e., $T_{\varphi_{\omega_a}}^*C_aT_{\varphi_{\omega_a}}=-C_a.$ 
\end{rem}

Denote $A$ the operator $A:=W_{\omega_a}C_aW_{\omega_a}=T_{f_a}C_0=C_0T_{1/f_a}$ from the previous section. In \cite{disco y composicion} Lemma  4.1 (or perform the elementary computation) it was shown that for $b\in\mathbb{D}$
\begin{equation}\label{shift chueco}
W_bT_{\varphi_b}W_b=S=T_z.
\end{equation}
We can combine this with Proposition \ref{anticonmutacion} to obtain
\begin{coro}
With the current notations, we have that
$$ 
SA+AS=0
$$
and
$$
SA^*+A^*S=2\omega_a\langle \ \ , k_{\omega_a}\rangle 1.
$$
\end{coro}
\begin{proof}
$$
0=W_{\omega_a}\left(T_{\varphi_{\omega_a}}C_a+C_aT_{\varphi_{\omega_a}}\right)W_{\omega_a}=W_{\omega_a}T_{\varphi_{\omega_a}}W_{\omega_a}W_{\omega_a}C_aW_{\omega_a}+W_{\omega_a}C_aW_{\omega_a}W_{\omega_a}T_{\varphi_{\omega_a}}W_{\omega_a}
$$
$$
=SA+AS.
$$
Similarly
$$
SA^*+A^*S=W_{\omega_a}\left(2\omega_a\langle \ \ , 1\rangle k_{\omega_a}\right)W_{\omega_a}=2\omega_a\langle \ \ , W_{\omega_a}1\rangle W_{\omega_a}k_{\omega_a}.
$$
The proof follows noting that $W_{\omega_a}1=\sqrt{1-|\omega_a|^2}T_{k_{\omega_a}}C_{\omega_a}1=\sqrt{1-|\omega_a|^2}k_{\omega_a}$ and thus 
$$
1=W_{\omega_a}W_{\omega_a}1=W_{\omega_a}(\sqrt{1-|\omega_a|^2}k_{\omega_a})=\sqrt{1-|\omega_a|^2}W_{\omega_a}k_{\omega_a},
$$
i.e., $W_{\omega_a}k_{\omega_a}=\frac{1}{\sqrt{1-|\omega_a|^2}}1$. 
\end{proof}
Note that the second assertion of the above corollary  is equivalent to 
\begin{equation}
S^*A+AS^*=2\bar{\omega}_a(\langle \ \ , k_{\omega_a}\rangle 1)^*=2\bar{\omega}_a\langle \ \ , 1\rangle k_{\omega_a}.
\end{equation}
Then, using the again this first assertion, 
$$
SS^*A+SAS^*=SS^*A-ASS^*=2\bar{\omega}_a\langle \ \ , 1\rangle Sk_{\omega_a}.
$$
Since $SS^*=I-\langle \ \ , 1\rangle 1$, we get that
$$
A\langle \ \ , 1\rangle 1-\langle \ \ , 1\rangle 1 A=\langle \ \ , 1\rangle A1-\langle \ \ , A1\rangle 1=2\bar{\omega}_a\langle \ \ , 1\rangle Sk_{\omega_a}.
$$
Since $A1=T_{f_a}C_01=f_a$, we obtain
\begin{equation}
\langle \ \ , f_a\rangle 1- \langle \ \ , 1\rangle f_a=2\bar{\omega}_a\langle \ \ , 1\rangle Sk_{\omega_a}.
\end{equation}

E. Andruchow, {\sc  {Instituto Argentino de Matem\'atica, `Alberto P. Calder\'on', CONICET, Saavedra 15 3er. piso,
(1083) Buenos Aires, Argentina }} and {\sc Universidad Nacional de General Sarmiento, J.M. Gutierrez 1150, (1613) Los Polvorines, Argentina}

  e-mail: eandruch@campus.ungs.edu.ar

\end{document}